\documentclass{article}%
\usepackage{amsfonts}
\usepackage{amsmath}
\usepackage{amssymb}
\usepackage{graphicx}%
\providecommand{\U}[1]{\protect\rule{.1in}{.1in}}
\newtheorem{theorem}{Theorem}

\newtheorem{corollary}[theorem]{Corollary}

\newtheorem{proposition}[theorem]{Proposition}

\newenvironment{proof}[1][Proof]{\noindent\textbf{#1.} }{\ \rule{0.5em}{0.5em}}
\begin{document}

\date{}
\title{A closed formula for the generating function of $p$-Bernoulli numbers}
\author{Levent Karg\i n$^{1}$ and Mourad Rahmani$^{2}$\\$^{1}$Akseki Vocational School, Alanya Alaaddin Keykubat University,\\Antalya TR-07630, Turkey\\$^{2}$USTHB, Faculty of Mathematics, P. O. Box32, El Alia, \\Bab Ezzouar, 16111, Algiers, Algeria\\leventkargin48@gmail.com and mrahmani@usthb.dz}
\maketitle

\begin{abstract}
In this paper, using geometric polynomials, we obtain a generating function of
$p$-Bernoulli numbers. As a consequences this generating function, we derive
closed formulas for the finite summation of Bernoulli and harmonic numbers
involving Stirling numbers of the second kind.



\end{abstract}

\section{Introduction}

Rahmani \cite{Rahmani} introduced $p$-Bernoulli numbers by constructing an
infinite matrix as follows:

The first row of the matrix $B_{0,p}=1$ and each entry $B_{n,p}$ is given
recursively by
\[
B_{n+1,p}=pB_{n,p}-\frac{\left(  p+1\right)  ^{2}}{p+2}B_{n,p+1}.
\]
The first column of the matrix $B_{n,0}=B_{n}$. Here, $B_{n}$ is the $n$th
Bernoulli number.

For every integer $p\geq-1$, these numbers have an explicit formula
\begin{equation}
B_{n,p}=\sum_{k=0}^{n}\left(  -1\right)  ^{k}%
\genfrac{\{}{\}}{0pt}{}{n}{k}%
\binom{k+p+1}{k}^{-1}k!, \label{1}%
\end{equation}
and are closely related to Bernoulli numbers by the following formula
\begin{equation}
B_{n,p}=\frac{p+1}{p!}\sum_{j=0}^{p}\left(  -1\right)  ^{j}%
\genfrac{[}{]}{0pt}{}{p}{j}%
B_{n+j}, \label{11}%
\end{equation}
where $%
\genfrac{[}{]}{0pt}{}{n}{k}%
$ is the Stirling number of the first kind \cite{Graham}. The first few
generating functions for $B_{n,p}$ ($p=1,2$) are
\begin{align*}
\sum_{n=0}^{\infty}B_{n,1}\frac{t^{n}}{n!}  &  =\frac{2\left[  \left(
t-1\right)  e^{t}+1\right]  }{\left(  e^{t}-1\right)  ^{2}},\\
\sum_{n=0}^{\infty}B_{n,2}\frac{t^{n}}{n!}  &  =\frac{3\left[  \left(
2t-3\right)  e^{2t}+4e^{t}-1\right]  }{2\left(  e^{t}-1\right)  ^{3}}.
\end{align*}

The main purpose of this study is to give a close form of the above results
as
\begin{equation}
\sum_{n=0}^{\infty}B_{n,p}\frac{t^{n}}{n!}=\frac{\left(  p+1\right)  \left(
t-H_{p}\right)  e^{pt}}{\left(  e^{t}-1\right)  ^{p+1}}+\left(  p+1\right)
\sum_{k=1}^{p}\binom{p}{k}\frac{H_{k}}{\left(  e^{t}-1\right)  ^{k+1}},
\label{4}%
\end{equation}
where $H_{n}$ is the harmonic numbers, defined by \cite[p. 258]{Graham}%
\[
H_{n}=\sum_{j=1}^{n}\frac{1}{j}.
\]
As a consequences of (\ref{4}), we have closed formulas for the finite
summation of Bernoulli and harmonic numbers.

For the proof of (\ref{4}), we use some properties of geometric polynomials.
The geometric polynomials are defined by means of the following generating
function \cite{T}%
\begin{equation}
\frac{1}{1-x\left(  e^{t}-1\right)  }=\sum_{n=0}^{\infty}w_{n}\left(
x\right)  \frac{t^{n}}{n!}\text{,} \label{10}%
\end{equation}
and have the explicit formula
\begin{equation}
w_{n}\left(  x\right)  =\sum_{k=0}^{n}%
\genfrac{\{}{\}}{0pt}{}{n}{k}%
k!x^{k}, \label{2}%
\end{equation}
where $%
\genfrac{\{}{\}}{0pt}{}{n}{k}%
$ is the Stirling number of the second kind \cite{Graham}. The Stirling
numbers of the second kind are defined by means of the following generating
function
\begin{equation}
\sum_{n=0}^{\infty}%
\genfrac{\{}{\}}{0pt}{}{n}{k}%
\frac{t^{n}}{n!}=\frac{\left(  e^{t}-1\right)  ^{k}}{k!}. \label{6}%
\end{equation}

Some other properties of geometric polynomials can be found in \cite{B, B3,
B4, BD, Diletal, Kargin}.

\section{A new generating function for $p$-Bernoulli numbers}

In this section, the main theorem and its applications are given.

Now, we give the main theorem of this paper.

\begin{theorem}
\label{teo1}For $p\geq0$, the following generating function holds true:%
\begin{equation}
\sum_{n=0}^{\infty}B_{n,p}\frac{t^{n}}{n!}=\frac{\left(  p+1\right)  \left(
t-H_{p}\right)  e^{pt}}{\left(  e^{t}-1\right)  ^{p+1}}+\left(  p+1\right)
\sum_{k=1}^{p}\binom{p}{k}\frac{H_{k}}{\left(  e^{t}-1\right)  ^{k+1}}.
\label{5}%
\end{equation}

\end{theorem}

For the proof of main theorem, we first need the following proposition.

\begin{proposition}
For $n>p\geq0$, we have
\begin{equation}
\frac{\left(  -1\right)  ^{p}}{\left(  p+1\right)  !}B_{n,p}%
=\underset{p+1\text{ times}}{\underbrace{%
{\displaystyle\int\limits_{-1}^{0}}
{\displaystyle\int\limits_{0}^{x_{p}}}
\ldots%
{\displaystyle\int\limits_{0}^{x_{3}}}
{\displaystyle\int\limits_{0}^{x_{2}}}
}}w_{n}\left(  x_{1}\right)  dx_{1}dx_{2}\ldots dx_{p-1}dx_{p}. \label{3}%
\end{equation}

\end{proposition}

\begin{proof}
If we integrate both sides of (\ref{2}) with respect to $x_{1}$ from $0$ to
$x_{2}$, we have%
\[%
{\displaystyle\int\limits_{0}^{x_{2}}}
w_{n}\left(  x_{1}\right)  dx_{1}=\sum_{k=0}^{n}%
\genfrac{\{}{\}}{0pt}{}{n}{k}%
k!\frac{x_{2}^{k+1}}{k+1}.
\]
Integrating both sides of the above equation with respect to $x_{2}$ from $0$
to $x_{3}$, we obtain%
\[%
{\displaystyle\int\limits_{0}^{x_{3}}}
{\displaystyle\int\limits_{0}^{x_{2}}}
w_{n}\left(  x_{1}\right)  dx_{1}dx_{2}=\sum_{k=0}^{n}%
\genfrac{\{}{\}}{0pt}{}{n}{k}%
k!\frac{x_{3}^{k+2}}{\left(  k+1\right)  \left(  k+2\right)  }.
\]
Applying the same procedure for $p$ times yields
\[%
{\displaystyle\int\limits_{0}^{x_{p}}}
\ldots%
{\displaystyle\int\limits_{0}^{x_{3}}}
{\displaystyle\int\limits_{0}^{x_{2}}}
w_{n}\left(  x_{1}\right)  dx_{1}dx_{2}\ldots dx_{p-1}=\sum_{k=0}^{n}%
\genfrac{\{}{\}}{0pt}{}{n}{k}%
k!\frac{x_{p}^{k+p}}{\left(  k+1\right)  \cdots\left(  k+p\right)  }.
\]
Finally, integrating both sides of the above equation with respect to $x_{p}$
from $-1$ to $0$ and using (\ref{1}) gives the desired equation.
\end{proof}

We note that taking $p=0$ in (\ref{3}) gives \cite[Theorem 1.2]{KELLER}.

Now, we are ready to give the proof of the main theorem.

\begin{proof}
[Proof of Theorem \ref{teo1}]Multiplying both sides of (\ref{3}) with
$\frac{t^{n}}{n!}$ and summing over $n$ from $0$ to infinitive, we have
\begin{align*}
\frac{\left(  -1\right)  ^{p}}{\left(  p+1\right)  !}\sum_{n=0}^{\infty
}B_{n,p}\frac{t^{n}}{n!}  &  =%
{\displaystyle\int\limits_{-1}^{0}}
{\displaystyle\int\limits_{0}^{x_{p}}}
\ldots%
{\displaystyle\int\limits_{0}^{x_{3}}}
\left[
{\displaystyle\int\limits_{0}^{x_{2}}}
\left(  \sum_{n=0}^{\infty}w_{n}\left(  x_{1}\right)  \frac{t^{n}}{n!}\right)
dx_{1}\right]  dx_{2}\ldots dx_{p-1}dx_{p}\\
&  =%
{\displaystyle\int\limits_{-1}^{0}}
{\displaystyle\int\limits_{0}^{x_{p}}}
\ldots%
{\displaystyle\int\limits_{0}^{x_{3}}}
\left[
{\displaystyle\int\limits_{0}^{x_{2}}}
\frac{1}{1-x_{1}\left(  e^{t}-1\right)  }dx_{1}\right]  dx_{2}\ldots
dx_{p-1}dx_{p}.
\end{align*}
If we evaluate the first integral, we obtain
\[%
{\displaystyle\int\limits_{0}^{x_{2}}}
\frac{1}{1-x_{1}\left(  e^{t}-1\right)  }dx_{1}=\frac{-1}{e^{t}-1}\ln\left(
1-x_{2}\left(  e^{t}-1\right)  \right)  .
\]
For the second time, we evaluate
\begin{align*}
&  \frac{-1}{e^{t}-1}%
{\displaystyle\int\limits_{0}^{x_{3}}}
\ln\left(  1-x_{2}\left(  e^{t}-1\right)  \right)  dx_{2}\\
&  \quad=\frac{1}{\left(  e^{t}-1\right)  ^{2}}\left[  \left(  1-x_{3}\left(
e^{t}-1\right)  \right)  \ln\left(  1-x_{3}\left(  e^{t}-1\right)  \right)
-\left(  1-x_{3}\left(  e^{t}-1\right)  \right)  +1\right]  .
\end{align*}
By induction on $p$, let us assume that the following equation holds%
\begin{align}
&
{\displaystyle\int\limits_{0}^{x_{p}}}
\ldots%
{\displaystyle\int\limits_{0}^{x_{3}}}
{\displaystyle\int\limits_{0}^{x_{2}}}
w_{n}\left(  x_{1}\right)  dx_{1}dx_{2}\ldots dx_{p-1}\label{12}\\
&  \quad=\frac{\left(  -1\right)  ^{p}}{\left(  e^{t}-1\right)  ^{p}\left(
p-1\right)  !}\left(  1-x_{p}\left(  e^{t}-1\right)  \right)  ^{p-1}\ln\left(
1-x_{p}\left(  e^{t}-1\right)  \right) \nonumber\\
&  \quad\quad\quad-\frac{\left(  -1\right)  ^{p}}{\left(  e^{t}-1\right)
^{p}\left(  p-1\right)  !}\left[  H_{p-1}\left(  1-x_{p}\left(  e^{t}%
-1\right)  \right)  ^{p-1}+H_{p-1}\right] \nonumber\\
&  \quad\quad\quad+\sum_{k=1}^{p-2}\frac{\left(  -1\right)  ^{p-k}%
H_{p-1-k}x_{p}^{k}}{\left(  e^{t}-1\right)  ^{p-k}\left(  p-1-k\right)
!k!}.\nonumber
\end{align}
Now, we want to prove that (\ref{12}) holds for the case $p+1.$ Let us
integrate both sides of (\ref{12}) with respect to $x_{p}$ from $0$ to $y.$
Then we have%
\begin{align*}
&
{\displaystyle\int\limits_{0}^{y}}
{\displaystyle\int\limits_{0}^{x_{p}}}
\ldots%
{\displaystyle\int\limits_{0}^{x_{3}}}
{\displaystyle\int\limits_{0}^{x_{2}}}
w_{n}\left(  x_{1}\right)  dx_{1}dx_{2}\ldots dx_{p-1}dx_{p}\\
&  \quad=\frac{\left(  -1\right)  ^{p}}{\left(  e^{t}-1\right)  ^{p}\left(
p-1\right)  !}%
{\displaystyle\int\limits_{0}^{y}}
\left(  1-x_{p}\left(  e^{t}-1\right)  \right)  ^{p-1}\ln\left(
1-x_{p}\left(  e^{t}-1\right)  \right)  dx_{p}\\
&  \quad\quad-\frac{\left(  -1\right)  ^{p}H_{p-1}}{\left(  e^{t}-1\right)
^{p}\left(  p-1\right)  !}%
{\displaystyle\int\limits_{0}^{y}}
\left(  1-x_{p}\left(  e^{t}-1\right)  \right)  ^{p-1}dx_{p}+\frac{\left(
-1\right)  ^{p}H_{p-1}}{\left(  e^{t}-1\right)  ^{p}\left(  p-1\right)  !}%
{\displaystyle\int\limits_{0}^{y}}
dx_{p}\\
&  \quad\quad+\sum_{k=1}^{p-2}\frac{\left(  -1\right)  ^{p-k}H_{p-1-k}%
}{\left(  e^{t}-1\right)  ^{p-k}\left(  p-1-k\right)  !k!}%
{\displaystyle\int\limits_{0}^{y}}
x_{p}^{k}dx_{p}.
\end{align*}
The first integral in the right hand-side equals%
\begin{align}
&  \frac{\left(  -1\right)  ^{p}}{\left(  e^{t}-1\right)  ^{p}\left(
p-1\right)  !}%
{\displaystyle\int\limits_{0}^{y}}
\left(  1-x_{p}\left(  e^{t}-1\right)  \right)  ^{p-1}\ln\left(
1-x_{p}\left(  e^{t}-1\right)  \right)  dx_{p}\label{13}\\
&  \quad=\frac{\left(  -1\right)  ^{p+1}}{\left(  e^{t}-1\right)  ^{p+1}%
p!}\left[  \left(  1-y\left(  e^{t}-1\right)  \right)  ^{p}\ln\left(
1-y\left(  e^{t}-1\right)  \right)  -\frac{\left(  1-y\left(  e^{t}-1\right)
\right)  ^{p}}{p}+\frac{1}{p}\right]  .\nonumber
\end{align}
For the second integral in the right hand-side, we obtain%
\begin{align}
&  \frac{\left(  -1\right)  ^{p}H_{p-1}}{\left(  e^{t}-1\right)  ^{p}\left(
p-1\right)  !}%
{\displaystyle\int\limits_{0}^{y}}
\left(  1-x_{p}\left(  e^{t}-1\right)  \right)  ^{p-1}dx_{p}\label{14}\\
&  \quad=\frac{\left(  -1\right)  ^{p+1}H_{p-1}}{\left(  e^{t}-1\right)
^{p+1}p!}\left(  1-y\left(  e^{t}-1\right)  \right)  ^{p}-\frac{\left(
-1\right)  ^{p+1}H_{p-1}}{\left(  e^{t}-1\right)  ^{p+1}p!}.\nonumber
\end{align}
For the third and fourth integrals, we find%
\begin{equation}
\frac{\left(  -1\right)  ^{p}H_{p-1}}{\left(  e^{t}-1\right)  ^{p}\left(
p-1\right)  !}%
{\displaystyle\int\limits_{0}^{y}}
dx_{p}=\frac{\left(  -1\right)  ^{p+1}H_{p-1}}{\left(  e^{t}-1\right)
^{p}\left(  p-1\right)  !}y \label{15}%
\end{equation}
and
\begin{equation}
\sum_{k=1}^{p-2}\frac{\left(  -1\right)  ^{p-k}H_{p-1-k}}{\left(
e^{t}-1\right)  ^{p-k}\left(  p-1-k\right)  !k!}%
{\displaystyle\int\limits_{0}^{y}}
x_{p}^{k}dx_{p}=\sum_{k=2}^{p-1}\frac{\left(  -1\right)  ^{p+1-k}H_{p-k}y^{k}%
}{\left(  e^{t}-1\right)  ^{p+1-k}\left(  p-k\right)  !k!}, \label{16}%
\end{equation}
respectively. Combining (\ref{13}), (\ref{14}), (\ref{15}) and (\ref{16}), we
achieve that
\begin{align*}
&
{\displaystyle\int\limits_{0}^{y}}
{\displaystyle\int\limits_{0}^{x_{p}}}
\ldots%
{\displaystyle\int\limits_{0}^{x_{3}}}
{\displaystyle\int\limits_{0}^{x_{2}}}
w_{n}\left(  x_{1}\right)  dx_{1}dx_{2}\ldots dx_{p-1}dx_{p}\\
&  \quad=\frac{\left(  -1\right)  ^{p+1}}{\left(  e^{t}-1\right)  ^{p+1}%
p!}\left(  1-y\left(  e^{t}-1\right)  \right)  ^{p}\ln\left(  1-y\left(
e^{t}-1\right)  \right) \\
&  \quad-\frac{\left(  -1\right)  ^{p+1}}{\left(  e^{t}-1\right)  ^{p+1}%
p!}\left[  H_{p}\left(  1-y\left(  e^{t}-1\right)  \right)  ^{p}+H_{p}\right]
\\
&  \quad+\sum_{k=1}^{p-1}\frac{\left(  -1\right)  ^{p+1-k}H_{p-k}y^{k}%
}{\left(  e^{t}-1\right)  ^{p+1-k}\left(  p-k\right)  !k!}.
\end{align*}
Finally, setting $y=-1$ in the above equation and using (\ref{3}), we arrive
at the desired equation.
\end{proof}

As an application of Theorem \ref{teo1}, we give the following theorem.

\begin{theorem}
\label{teo2}For $n>p\geq0$, we have%
\begin{equation}
\sum_{k=p+1}^{n}\binom{n}{k}%
\genfrac{\{}{\}}{0pt}{}{k}{p+1}%
B_{n-k,p}=\frac{p^{n-1}\left(  n-pH_{p}\right)  }{p!}+\sum_{j=1}^{p}%
\genfrac{\{}{\}}{0pt}{}{n}{p-j}%
\frac{H_{j}}{j!}. \label{9}%
\end{equation}

\end{theorem}

\begin{proof}
Multiplying both sides of (\ref{5}) with $\frac{\left(  e^{t}-1\right)
^{p+1}}{\left(  p+1\right)  !}$ and using (\ref{6}), the left hand side of
(\ref{5}) becomes
\begin{align}
\frac{\left(  e^{t}-1\right)  ^{p+1}}{\left(  p+1\right)  !}\sum_{n=0}%
^{\infty}B_{n,p}\frac{t^{n}}{n!}  &  =\sum_{n=0}^{\infty}\sum_{k=0}^{\infty}%
\genfrac{\{}{\}}{0pt}{}{k}{p+1}%
\frac{B_{n,p}}{k!n!}t^{n+k}\label{8}\\
&  =\sum_{n=0}^{\infty}\left[  \sum_{k=0}^{n}\binom{n}{k}%
\genfrac{\{}{\}}{0pt}{}{k}{p+1}%
B_{n-k,p}\right]  \frac{t^{n}}{n!}.\nonumber
\end{align}
For the right hand side of (\ref{5}), we obtain%
\begin{align}
&  \frac{te^{pt}}{p!}-\frac{H_{p}e^{pt}}{p!}+\sum_{k=1}^{p}\frac{H_{k}}%
{k!}\frac{\left(  e^{t}-1\right)  ^{p-k}}{\left(  p-k\right)  !}\label{7}\\
&  \qquad\qquad=\sum_{n=1}^{\infty}\left[  \frac{p^{n-1}\left(  n-pH_{p}%
\right)  }{p!}+\sum_{k=1}^{p}\frac{H_{k}}{k!}%
\genfrac{\{}{\}}{0pt}{}{n}{p-k}%
\right]  \frac{t^{n}}{n!}.\nonumber
\end{align}
Finally, comparing the coefficients of $\frac{t^{n}}{n!}$ in (\ref{8}) and
(\ref{7}) completes the proof.
\end{proof}

Using (\ref{11}) in Theorem \ref{teo2} gives the following corollary.

\begin{corollary}
\label{cor1}For $n>p\geq0$,%
\[
\sum_{k=p+1}^{n}\sum_{j=0}^{p}\binom{n}{k}%
\genfrac{\{}{\}}{0pt}{}{k}{p+1}%
\genfrac{[}{]}{0pt}{}{p}{j}%
\left(  -1\right)  ^{j}B_{n+j-k}=\frac{p^{n-1}\left(  n-pH_{p}\right)  }%
{p+1}+\frac{p!}{p+1}\sum_{j=1}^{p}%
\genfrac{\{}{\}}{0pt}{}{n}{p-j}%
\frac{H_{j}}{j!}.
\]

\end{corollary}

As a consequences of Corollary \ref{cor1}, the following sums are obtained
$\left(  p=1,2\right)  :$%

\begin{align*}
\sum_{k=2}^{n}\binom{n}{k}%
\genfrac{\{}{\}}{0pt}{}{k}{2}%
B_{n+1-k}  &  =\frac{-\left(  n-1\right)  }{2},\\
\sum_{k=3}^{n}\binom{n}{k}%
\genfrac{\{}{\}}{0pt}{}{k}{3}%
\left(  B_{n+2-k}-B_{n+1-k}\right)   &  =\frac{2^{n-1}\left(  n-3\right)
+2}{3}.
\end{align*}

Setting $n=p+1$ in Theorem \ref{teo2} and using $B_{0,p}=1$, we arrive at the
following corollary.

\begin{corollary}
For $p\geq1$, we obtain a new closed formula for the finite summation of
harmonic numbers%
\[
\sum_{j=1}^{p}%
\genfrac{\{}{\}}{0pt}{}{p+1}{p-j}%
\frac{H_{j}}{j!}=\frac{p!-\left(  p+1\right)  p^{p}+p^{p+1}H_{p}}{p!}.
\]

\end{corollary}


\begin{thebibliography}{99}                                                                                               %


\bibitem {B}K. N. Boyadzhiev, A series transformation formula and related
polynomials, \textit{Int. J. Math. Math. Sci.} \textbf{23} (2005), 3849--3866.

\bibitem {B3}K. N. Boyadzhiev, Apostol-Bernoulli functions, derivative
polynomials and Eulerian polynomials, \textit{Adv. Appl. Discrete Math.}
\textbf{1/2} (2008), 109--122.

\bibitem {B4}K. N. Boyadzhiev, Close encounters with the Stirling numbers of
the second kind, \textit{Math. Mag.} \textbf{85} (2012), 252--266.

\bibitem {BD}K. N. Boyadzhiev and A. Dil, Geometric polynomials: properties
and applications to series with zeta values, \textit{Anal. Math.},
\textbf{42/3} (2016), 203--224.

\bibitem {Diletal}A. Dil and V. Kurt, Investigating geometric and exponential
polynomials with Euler-Seidel matrices, \textit{J. Integer Seq.} \textbf{14}
(2011), Article 11.4.6.

\bibitem {Graham}R. L. Graham, D. E. Knuth and O. Patashnik, \textit{Concrete
Mathematics}, Addison-Wesley Publ. Com., New York, 1994.

\bibitem {Kargin}L. Karg\i n, Some formulae for products of Fubini polynomials
with applications, arXiv:1701.01023.

\bibitem {KELLER}B. C. Kellner, Identities between polynomials related to
Stirling and harmonic numbers, \textit{Integers} \textbf{14 }(2014), \#A54.

\bibitem {Rahmani}M. Rahmani, On $p$-Bernoulli numbers and polynomials,
Journal of Number Theory 157 (2015) 350--366.

\bibitem {T}S. M. Tanny, On some numbers related to the Bell numbers,
\textit{Canad. Math. Bull.} \textbf{17} (1974), 733--738.
\end{thebibliography}
\end{document}